\documentclass[a4paper,10pt]{amsart}
\usepackage[centertags]{amsmath}
\usepackage{amsfonts}
\usepackage{amssymb}
\usepackage{amsthm}
\usepackage[usenames,dvipsnames]{color}
\usepackage{newlfont}
\usepackage{graphicx}
\usepackage{epstopdf}
\usepackage{cite}
\usepackage{bm}
\newcommand{\Z}{\mathbb{Z}}

\newcommand{\bq }{\begin{equation}}
\newcommand{\eq }{\end{equation}}
\theoremstyle{plain}
\newtheorem{thm}{Theorem}[section]
\newtheorem{lem}[thm]{Lemma}
\newtheorem{prop}[thm]{Proposition}

\newtheorem{cor}[thm]{Corollary}
\newtheorem{rem}[thm]{Remark}
\newtheorem{preg}[thm]{Question}

\theoremstyle{definition}
\newtheorem{defn}[thm]{Definition}

\theoremstyle{example}

\title{On the locus of smooth plane curves with a fixed automorphism group}

\author[E. Badr] {Eslam Badr}
\address{$\bullet$\,\,Eslam Badr}
\address{Departament Matem\`atiques, Edif. C, Universitat Aut\`onoma de Barcelona\\
08193 Bellaterra, Catalonia}
\email{eslam@mat.uab.cat}
\address{Department of Mathematics,
Faculty of Science, Cairo University, Giza-Egypt}
\email{eslam@sci.cu.edu.eg}

\author[F. Bars] {Francesc Bars}
\address{$\bullet$\,\,Francesc Bars}
\address{Departament Matem\`atiques, Edif. C, Universitat Aut\`onoma de Barcelona\\
08193 Bellaterra, Catalonia} \address{{Barcelona
Graduate School of Mathematics, Catalonia}}
 \email{francesc@mat.uab.cat}
\thanks{E. Badr and F. Bars are supported by MTM2013-40680-P}



\begin{document}
\begin{abstract}
Let $M_g$ be the moduli space of smooth, genus $g$ curves over an
algebraically closed field $K$ of zero characteristic. Denote by
${M_g(G)}$ the subset of $M_g$ of curves $\delta$ such that $G$ (as
a finite non-trivial group) is isomorphic to a subgroup of
$Aut(\delta)$, the full automorphism group of $\delta$, and let
$\widetilde{M_g(G)}$ be the subset of curves $\delta$ such that
$G\cong Aut(\delta)$. Now, for an integer $d\geq 4$, let $M_g^{Pl}$
be the subset of $M_g$ representing smooth, genus $g$ plane curves
of degree $d$ (in such case, $g=(d-1)(d-2)/2$), and consider the sets
$M_g^{Pl}(G):=M_g^{Pl}\cap M_g(G)$ and
$\widetilde{M_g^{Pl}(G)}:=\widetilde{M_g(G)}\cap M_g^{Pl}$.
\par In this paper, we study some aspects of the irreducibility of
$\widetilde{M_g^{Pl}(G)}$ and its interrelation with the existence
of ``normal forms'', i.e. non-singular plane equations (depending on
a set of parameters) such that a specialization of the parameters
gives a certain non-singular plane model associated to the elements
of $\widetilde{M_g^{Pl}(G)}$. In particular, we introduce the
concept of being equation strongly irreducible (ES-Irreducible) for
which the locus $\widetilde{M_g^{Pl}(G)}$ is represented by a single
``normal form''. Henn, in \cite{He}, and Komiya-Kuribayashi, in
\cite{KK}, observed that $\widetilde{M_3^{Pl}(G)}$ is
ES-Irreducible. In this paper we prove that this phenomena does not
occur for any odd $d>4$. More precisely, let $\Z/m\Z$ be the cyclic
group of order $m$, we prove that $\widetilde{M_g^{Pl}(\Z/(d-1)\Z)}$ is not ES-Irreducible for any odd integer $d\geq5$, and the
number of its irreducible components is at least two.
Furthermore, we conclude the previous result when $d=6$ for the
locus $\widetilde{M_{10}^{Pl}(\Z/3\Z)}$.
\par Lastly, we prove the analogy of these statements when $K$ is any algebraically closed field of positive
characteristic $p$ such that $p>(d-1)(d-2)+1$.
\end{abstract}

\maketitle

\section{Introduction}

Let $K$ be an algebraically closed field of zero characteristic and
fix an integer $d\geq 4$. We consider, up to $K$-isomorphism, a
projective non-singular curve $\delta$ of genus $g=(d-1)(d-2)/2$, and
assume that $\delta$ has a non-singular plane model, i.e. $\delta\in
M_g^{Pl}$.

{It is well known that} any $\delta\in M_g^{Pl}(G)$ corresponds to a
set {$\{C_{\delta}\}$ }of non-singular plane models in
$\mathbb{P}^2(K)$ such that any two of them are {$K$-isomorphic
through a projective transformation $P\in PGL_3(K)$} (where
$PGL_N(K)$ is the classical projective linear group of $N\times N$
invertible matrices over $K$), and their automorphism groups are
conjugate. More concretely, {fixing a non-singular, degree $d$ plane
model $C$ of $\delta$ whose defining equation is $F(X;Y;Z)=0$}.
Then, $Aut(C)$ is a finite subgroup of $PGL_3(K)$, and also we have
$\rho(G)\preceq Aut(C)$ for some injective representation
$\rho:G\hookrightarrow PGL_3(K)$. Moreover, $\rho(G)=Aut(C)$
whenever $\delta\in\widetilde{M_g^{Pl}(G)}$. For another
non-singular plane model $C'$ of $\delta$, there exists $P\in
PGL_3(K)$ where $C'$ is defined by $F(P(X,Y,Z))=0$, and
$P^{-1}\rho(G)P\preceq Aut(C')$ (respectively,
$P^{-1}\rho(G)P=Aut(C')$ if $\delta\in\widetilde{M_g^{Pl}(G)}$) .

\par For an arbitrary, but a fixed degree $d$, classical and deep questions on the subject are: list the
groups that appear as the exact automorphism groups of algebraic
non-singular plane curves of degree $d$, and for each such group,
determine the associated ``normal forms'', i.e. a finite set of
homogenous equations $\{N_{1,G},\ldots,N_{k,G}\}$ in $X,Y,Z$
together with some parameters (under some restrictions) such that
any specialization of a certain $N_{i,G}$ in $K$ corresponds to a
unique $\delta\in\widetilde{M_g^{Pl}(G)}$ (is the one that is
associated to the non-singular plane model given by the
specialization of the normal form $N_{i,G}$), and given any
$\delta\in\widetilde{M_g^{Pl}(G)}$, there exists a unique
$i_{\delta}$ and a specialization of the parameters in $K$ for
$N_{i_{\delta},G}$ such that one obtains a plane non-singular model
associated to $\delta$.
%
%
 In particular, any specialization of the
parameters of two distinct $N_{i,G}$ gives two non-singular plane
models, which in turns relate to two non-isomorphic plane
non-singular curves of $\widetilde{M_g^{Pl}(G)}$.


For $d=4$, Henn in \cite{He} and Komiya-Kuribayashi in \cite{KK},
answered the above natural questions. See also Lorenzo's thesis
\cite{Lo} \S\,2.1 and \S\,2.2,  in order to fix some minor details.
It appears, for $d=4$, the following phenomena: any element of
$\widetilde{M_3^{Pl}(G)}$ has a non-singular plane model through
some specialization of the parameters of a single normal form.
If this phenomena appears for some $g$, we say that the locus
$\widetilde{M_g^{Pl}(G)}$ is ES-Irreducible (see \S2 for the precise
definition). This is a weaker condition than the irreducibility of
this locus inside the moduli space $M_g$. In particular, it
follows, by Henn \cite{He} and Komiya-Kuribayashi \cite{KK}, that the
locus $\widetilde{M_3^{Pl}(G)}$ is always ES-Irreducible.

The motivation of this work is that we did not expect
$\widetilde{M_g^{Pl}(G)}$ to be ES-Irreducible in general. In order
to construct counter examples for which $\widetilde{M_g^{Pl}(G)}$ is
not ES-Irreducible: we need first, a group $G$ such that there exist
at least two injective representations $\rho_i:G\hookrightarrow
PGL_3(K)$ with $i=1,2$, which are not conjugate (i.e\,\,there is no
transformation $P\in PGL_3(K)$ with $P^{-1}\rho_1(G)P=\rho_2(G)$,
more details are included in \S2), and for the zoo of groups that
could appear for non-singular plane curves \cite{Harui}, we consider
$G$, a cyclic group of order $m$. Secondly, one needs to prove the
existence of two non-singular plane curves with automorphism groups conjugate to $\rho_i(G)$ for each $i=1,2$.

\par The main results of the paper are that, the locus
$\widetilde{M_g^{Pl}(\Z/(d-1)\Z)}$ is not ES-irreducible for any odd degree $d\geq 5$, and it has
at least two irreducible components (If d=5, we know by \cite{BaBaaut}, that the only group $G$ for which
$\widetilde{M_{6}^{Pl}(G)}$ is not ES-Irreducible is $\Z/4\Z$).
For $d$ even, we prove in section \S\,5 that
$\widetilde{M_{10}^{Pl}(\Z/3\Z)}$ is not ES-irreducible. It is to be
noted that we may conjecture, by our work in \cite{BaBacyc}, that
the locus $\widetilde{M_g^{Pl}(\Z/m\Z)}$ could not be ES-Irreducible
only if $m$ divides $d$ or $d-1$ (this is true at least until degree
$9$ by \cite{BaBacyc}). Concerning positive characteristic, in the
last section (\S\,6) of this paper, we prove that the above examples
of non-irreducible loci are also valid when $K$ is an algebraically
closed field of positive characteristic $p>0$, provided that the
characteristic $p$ is big enough, once we fix the degree $d$.

The irreducibility of the loci $\widetilde{M_g^{Pl}(\Z/m)}$ seems to
be very deep problem. In \S 2, we give some insights that relate the
above locus with subsets in classical loci of the moduli space. In
particular, with the loci of curves of genus $g$ with a prescribed
cyclic Galois subcover. As an explicit example, we
deal with the question for the locus $\widetilde{M_6^{Pl}(\Z/8)}$,
which is ES-Irreducible, and is represented by a single normal form
with only one paramater. In \cite{BaBacyc}, we proved that
${M_g^{Pl}(G)}$ is irreducible when $G$ has an element of order
$(d-1)^2$, $d(d-1)$, $d(d-2)$ or $d^2-3d+3$, since this locus has
only one element. In particular, we prove in \cite{BaBacyc} that
$\widetilde{M_g^{Pl}(\Z/d(d-1))}$ and
$\widetilde{M_g^{Pl}(\Z/(d-1)^2)}$ are irreducible.

\subsection*{Acknowledgments} It is our pleasure to express our sincere gratitude to Xavier
Xarles and Joaquim Ro\'e for their suggestions. We also thank
Massimo Giulietti and Elisa Lorenzo for noticing us about some
bibliography on automorphism of curves. We appreciate a lot the
comments and suggestions of a referee that improved the paper to a
great extent in its present form.

\section{On the {loci} $M_g^{Pl}(G)$ and $\widetilde{M_g^{Pl}(G)}$.}

Consider a non-singular projective curve $\delta$ of genus
$g:=\frac{(d-1)(d-2)}{2}\geq 2$ over $K$ with a non-trivial finite
subgroup $G$ of $Aut(\delta)$. We always assume that $\delta$
admits a non-singular plane equation, and {$\delta\in M_g^{Pl}(G)$}.

\par {The} linear systems $g^2_d$ are
unique, up to multiplication by $P\in PGL_3(K)$ in $\mathbb{P}^2(K)$
\cite[Lemma 11.28]{Book}{. Therefore,} we can take $C$ {as a
non-singular plane} model of $\delta$, which is {defined} by a
projective plane equation $F(X;Y;Z)=0${,} and $Aut(C)$ {as} a finite
subgroup of $PGL_3(K)$ that fixes the equation $F$ and is isomorphic
to $Aut(\delta)$. Any other plane model of $\delta$ is given by
$C_{P}:\,F(P(X;Y;Z))=0$ with $Aut(C_{P})=P^{-1}Aut(C)P$ for some
$P\in PGL_3(K)${,} and $C_{P}$ is $K$-equivalent or $K$-isomorphic
to $C$. In particular, for $\delta\in M_g^{Pl}(G)$, {there} exists
$\rho:G\hookrightarrow PGL_3(K)${,} where $\rho(G)\leq Aut(C)$ and
$P^{-1}\rho(G)P\leq Aut(C_P)$.

We denote by $\rho(M_g^{Pl}(G))$ the {locus of elements} $\delta\in
M_g^{Pl}(G)$ such that $G$ acts on {some} plane model associated to
$\delta$ {as} $P^{-1}\rho(G)P$ for {some} $P\in PGL_3(K)$, and
similarly for $\rho(\widetilde{M_g^{Pl}(G)})$. {Also, d}enote by $A_G$ the quotient set $\{\rho:G\hookrightarrow
PGL_3(K)\}/\sim${,} where $\rho_1\sim\rho_2$ if and only if $\exists
P\in PGL_3(K)$ such that $\rho_1(G)=P^{-1}\rho_2(G)P$.

Clearly $M_g^{Pl}(G)=\cup_{[\rho]\in A_G} \rho(M_g^{Pl}(G))${, where
$[\rho]$ denotes the class of $\rho$ in $A_G$}.

\begin{lem} The {locus} $\widetilde{M_g^{Pl}(G)}$ is the disjoint
union of $\rho(\widetilde{M_g^{Pl}(G)})$ where $[\rho]$ runs
{through $A_G$}.
\end{lem}
\begin{proof} {It is clear, by definition, that $\widetilde{M_g^{Pl}(G)}=\cup_{[\rho]\in A_G} \rho(\widetilde{M_g^{Pl}(G)})$.
Moreover,} $\delta\in \rho_1(\widetilde{M_g^{Pl}(G)})\cap
\rho_2(\widetilde{M_g^{Pl}(G)})$ means that it has a plane model $C$
{such that} $Aut(C)=P_1^{-1}\rho_1(G)P_1=P_2^{-1}\rho_2(G)P_2$ for
{some} $P_1,P_2\in PGL_3(K)${. T}herefore{,} $\rho_1\sim\rho_2$.
\end{proof}
\begin{rem} If $\delta\in\rho_1(M_g^{Pl}(G))\cap\rho_2(M_g^{Pl}(G))$
with $[\rho_1]\neq[\rho_2]\in A_G$, and {$C$ is} a plane model of
$\delta$, then $Aut(C)\leq PGL_3(K)$ should have two {non-conjugate
subgroups that are isomorphic to $G$}. A detailed study of the work
of Blichfeldt in \cite{Bli} would give the list of $G$ {for which}
the decomposition $M_g^{Pl}(G)=\cup_{[\rho]\in A_G}
\rho(M_g^{Pl}(G))$ may not be disjoint {(if any)}.
\end{rem}

Fix {a} $[\rho]\in A_G$ then we can associate {infinitely many
non-singular plane models to} $\delta \in \rho(M_g^{Pl}(G))$, which
are $K$-isomorphic through a change of variables $P\in PGL_3(K)${.}
But we can consider only the models such that $G$ is identified
{with} $\rho(G)\leq PGL_3(K)$ for some $\rho$ in $[\rho]\in A_G$
as{, the full} automorphism group. Under this restriction, $\delta$
can be associated with a non-empty 
family of non-singular models such that any two {of
them} are isomorphic, through a projective transformation $P$ {that
satisfies} $P^{-1}\rho(G)P=\rho(G)$.

 Recall that, it is a necessary condition{,} for a projective plane curve of degree $d$ to be non-singular{,} that the defining equation of any model has degree $\geq d-1$ in each variable. For a non-zero monomial $cX^iY^jZ^k$, {its exponent is defined to be} $max\{i,j,k\}$. For a homogeneous polynomial $F$, the core of $F$ is
defined as the sum of all terms of $F$ with the greatest exponent.
{Now,} we can assume, through a diagonal change of variables, that a
non-singular plane model {(whenever it exists) has only monic
monomials in its core}. Consequently, we reduce the {situation} to
{a} set of $K$-isomorphic non-singular plane models
${\{F_{C}(X;Y;Z)=0\}}$ associated to $\delta${, where $\rho(G)$
leaves invariant} the equation (being a subgroup of automorphisms of
such model){,} and each term of the core of {$F_{C}(X;Y;Z)$} is
monic.

\begin{lem} {Consider} $G$, a non-trivial finite group{, and an injective representation} $\rho:G\hookrightarrow
PGL_3(K)$ {of $G$} such that $\rho(M_g^{Pl}(G))$ is non-empty. There
exists a single normal form representing the {locus}
$\rho(M_g^{Pl}(G))$, i.e. {a} homogenous polynomial
$F_{\rho,G}(X;Y;Z)=0$ of degree $d$ in the variables $X,Y$ and $Z$,
{which is} endowed with {some} parameters as the coefficients of the
lower order terms ({together with} some restrictions){. M}ore
concretely, every specialization {in $K$} of the parameters of
$F_{\rho,G}$ ({with respect to} the restrictions on the parameters)
gives a {non-singular plane} model of an element ${\delta\in
}\rho(M_g^{Pl}(G))$, and viceversa {(any element
$\delta\in\rho(M_g^{Pl}(G))$ corresponds to some specialization in
$K$ of the parameters of $F_{\rho,G}$ such that the resulting plane
model of $\delta$ in $\mathbb{P}^2(K)$ is non-singular)}. A similar
statement holds for {the locus} $\rho(\widetilde{M_g^{Pl}(G)})${.
In} such case{,} we {call} $F_{\rho,G,*}$ a single normal form.
Moreover, {these} normal forms are unique up to change of the
variables $X,Y,Z$ by $P\in PGL_3(K)$.
\end{lem}
\begin{proof} Let $\sigma\in G$ be an automorphism of maximal order $m>1${,} and choose an element $\rho$ in $[\rho]\in {A_G}$ such that $\rho(\sigma)$ is diagonal of the form $diag(1,\xi_m^a,\xi_m^b)$
{where} $0\leq a<b${, and} $\xi_m$ a primitive $m$-th root of unity
in $K$. Following the same technique in \cite{Dolgachev} or {in}
\cite{BaBacyc} (for a general discussion), we can associate to the
set $\rho(M_g^{Pl}(<\sigma>)$ a non-singular plane equation $F_{m,
(a,b)}(X;Y;Z)$ with a certain set of parameters ({under some
algebraic} restrictions in order to ensure the non-singularity).
{This is a ``normal form'' for $\rho(M_g^{Pl}(<\sigma>))$, and it is
also unique(up to $K$-equivalence) by construction}. For example,
{if $0<a<b<m$ and all the reference points
$\{[1:0:0],[0:1:0],[0:0:1]\}$ satisfy the normal form
$F_{m,(a,b)}(X;Y;Z)=0$, we deduce that }$F_{m,(a,b)}(X;Y;Z)$ is $
X^{d-1}Y+Y^{d-1}Z+
Z^{d-1}X+\sum_{j=2}^{\lfloor\frac{d}{2}\rfloor}\,\,\big(X^{d-j}L_{j,X}+Y^{d-j}L_{j,Y}+Z^{d-j}L_{j,Z}\big),$
where $L_{j,B}$ is a homogenous polynomial of degree $j$ {in the
variables $\{X,Y,Z\}\setminus\{B\}$, and the parameters' list is
included} as the coefficients of {its} monomials. The first three
factors implies that $a\equiv(d-1)a+b\equiv(d-1)b\,(mod\,m)${. In
particular,} $m|d^2-3d+3$. The defining equation $F_{m,(a,b)}$ in
such situation, follows immediately by checking monomials'
invariance in each $L_{j,B}$. For example, rewrite $L_{j,X}$ as
$\sum_{i=0}^j \beta_{j,i}Y^iZ^{j-i}${, where $\beta_{j,i}$ are
parameters in $K$}. Then, $\beta_{j,i}=0$ if $m\nmid ai+(j-i)b$,
since $diag(1;\xi_m^a;\xi_m^b)\in Aut(C)$. Observe that{,} in order
to obtain such $F_{m,(a,b)}${,} we {choose a model $C$} for
any $\delta\in \rho(M_g^{Pl}(G))$ {that satisfies the condition}
$\rho(\sigma)=diag(1;\xi_m^a;\xi_m^b)$, and the monomials of the
core are monic. {Secondly, we impose that $<\sigma>\subseteq Aut(C)$
to get the required unique expression}.

Now, {to  move} from $\rho(<\sigma>)$ to $\rho(G)$, we assume {a}
generator $u_G$ of $G${,} which does not belong to $<\sigma>${,} and
{then we use the fact that} $\rho(u_G)$ {retains} invariant the
defining equation {$F_{m,(a,b)}=0$} {to obtain} some specific
algebraic relations between the parameters of $F_{m,(a,b)}${. T}hen,
$F_{\rho,G}$ is obtained from $F_{m,(a,b)}$ by {repeating the
procedure for each such generator {$u_G$} and imposing the algebraic
relations between the coefficients of the monomials (i.e. the
parameters' list)} of $F_{m,(a,b)}$. {By a similar argument, we
obtain $F_{\rho,G,*}$ from $F_{\rho,G}$. In fact, for a finite group
$H$ such that $\rho(G)\leq H\leq PGL_3(K)$, and for which there
exists a non-singular plane curve of genus $g$ whose automorphism
group is isomorphic to $H$, we need to apply the process above for
the generators of $H$ that are not in $\rho(G)$. Therefore, we only
need to consider a complement of certain algebraic constraints so
that $\delta$ does not have a bigger automorphism group {than} $H$.}

\end{proof}
\begin{rem} {It} could happen that two different
specializations of $F_{\rho,G}$ {in} $K$ {give non-singular plane}
models {that correspond to} the same curve $\delta\in
\rho(M_g^{Pl}(G))${. T}his happens if {there} exists {a
transformation} $P$ from one model to {the other with}
$P^{-1}\rho(G)P=\rho(G)${,} and
$P^{-1}\rho(<\sigma>)P=\rho(<\sigma>)$. {We can ensure that this
phenomena will not occur by assuming more restrictions to the
parameters of $F_{\rho,G}$, but we did not do this to our notion
``normal form''. These further assumptions became recently explicit
for the loci $\rho(M_3^{Pl}(G))$ through fixing the missing details
in the tables of Henn \cite{He}, by Lorenzo \cite{Lo}. We also
investigated such restrictions particularly for the loci
$\rho(M_6^{Pl}(\Z/8))$ and $\rho(\widetilde{M_6^{Pl}(\Z/8)})$ at the
end of this section}.
\end{rem}

It is difficult to determine the groups $G$ and $[\rho]\in A_G$ such
that $\rho(M_g^{Pl}(G))$ is non-empty for some fixed $g$. Henn
\cite{He} obtained this determination for $g=3$, Badr-Bars
\cite{BaBaaut} for $g=6${,} and for a general implementation of any
degree {$d\geq5$}, we refer to \cite{BaBacyc} (in which we formulate
an algorithm to determine the $\rho$'s when $G$ is cyclic).

\begin{defn} Write $M_g^{Pl}(G)$ as $\cup_{[\rho]\in A_G}\rho(M_g^{Pl}(G))$, we define the number
of the equation components of $M_g^{Pl}(G)$ to be the number of
elements $[\rho]\in A_G$ such that $\rho(M_g^{Pl}(G))$ is not empty.
We say that $M_g^{Pl}(G)$ is equation irreducible if {it is not
empty and} $M_g^{Pl}(G)=\rho(M_g^{Pl}(G))$ for a certain $[\rho]\in
A_G$. {Similar notion arises for the locus}
$\widetilde{M_g^{Pl}(G)}=\cup_{[\rho]\in
A_G}\rho(\widetilde{M_g^{Pl}(G)})${. W}e define the number of the
strongly equation irreducible components of
$\widetilde{M_g^{Pl}(G)}$ to be the number of the elements
$[\rho]\in A_G$ such that $\widetilde{\rho(M_g^{Pl}(G))}$ is not
empty. We say that $\widetilde{M_g^{Pl}(G)}$ is equation strongly
irreducible (or simply, ES-irreducible) if it is not empty and
$\widetilde{M_g^{Pl}(G)}=\widetilde{\rho(M_g^{Pl}(G))}$ for some
$[\rho]\in A_G$.
\end{defn}

Of course, if $\widetilde{M_g^{Pl}(G)}$ is not ES-irreducible then
it is not irreducible and the number of the strongly irreducible
equation components of $\widetilde{M_g^{Pl}(G)}$ is a lower bound
for the number of {its} irreducible components.

In this language, we can formulate the main result in \cite{He} as
follows:
\begin{thm}[Henn, Komiya-Kuribayashi] If $G$ is a non-trivial group that appears as the full
automorphism group of a non-singular plane  curve of degree $4$,
then $\widetilde{M_3^P(G)}$ is ES-Irreducible.
\end{thm}
\begin{rem} Henn in \cite{He}, observed that $M_3^{Pl}(\Z/3)$ {has already} two
irreducible equation components{. B}ut one of {these} components has
always a bigger automorphism group namely, $S_3$ the symmetry group
of of order 3.
\end{rem}

To finish this section, we state some natural questions concerning
the locus $\rho(M_g^{Pl}(G))${,} and similar questions can be
{formulated} for $\rho(\widetilde{M_g^{Pl}(G)})$ with different loci
{in the moduli space} of genus $g$ curves:
\begin{preg} Is it true{, for} all the elements of $\rho(M_g^{Pl}(G))${,} {that} the corresponding
Galois covers $\delta\rightarrow \delta/G$ have {a} fixed
ramification data?
\end{preg}
We believe that the answer to this question for $K=\mathbb{C}$ (i.e.
Riemann surfaces) should be always true from the work of Breuer
\cite{Bre}. See Remark \ref{rem2} for the explicit Galois subcover
and the ramification data {of} the locus
$\rho{(M_6^{Pl}(\Z/4\Z))}${,} and \S2.1 for the {locus}
$\rho(M_6^{Pl}(\Z/8\Z))$.

\begin{preg} Is $\rho(M_g^{Pl}(G))$ an irreducible set when $G$ is
a cyclic group?
\end{preg}
 It is to be noted that when $K=\mathbb{C}$, Cornalba \cite{Cor}, {for a cyclic group} $G$ of prime order,
 and Catanese \cite{Ca}, for general order, obtained that the locus
 of smooth projective {genus $g$ curves} with a cyclic Galois
 subcover of {a} group {that is} isomorphic to $G$ {and a} prescribed ramification is
 irreducible.

Concerning the irreducibility question, we prove in \cite{BaBacyc}
that if $G$ has an element of large order $(d-1)^2$, $d(d-1)$,
$d(d-2)$ or $d^2-3d+3$ then $\rho(M_g^{Pl}(G))$ has at most one
element{. T}herefore, is irreducible. At \S2.1, we deal {with the}
irreducibility {of} the ES-Irreducible {locus} $M_6^{Pl}(\Z/8\Z)${,}
where the single ``normal form'' has only one parameter.

Moreover, Catanese, L\"onne and Perroni in \cite[\S2]{CLP} define a
topological invariant for the loci $M_g(G)$, which is trivial if it
is irreducible.
\begin{preg}
Consider a non-trivial group {$G$ such that} the set $A_G$ is given
by one element (see {the} next section for {such groups}). Is it a
necessary condition {that the topological invariant in
\cite[\S2]{CLP} is trivial in order to be irreducible}? Is it true
that {the loci} $M_g^{Pl}(G)$ are irreducible?
\end{preg}

%
%
%

\subsection{The loci $M_6^{Pl}(\Z/8)$ and
$\widetilde{M_6^{Pl}(\Z/8\Z)}$} \mbox{}\\

Consider{, in the moduli} $M_6$, an element $\delta$ {that} has a
non-singular plane model with an effective action of the cyclic
group of order 8. {In other words,} $\delta\in M_6^{Pl}(\Z/8\Z)$.
Following \cite{BaBacyc}, \cite{Dolgachev} or the table {in} \S4
{of} this note, {one find that}
$M_6^{Pl}(\Z/8)=\rho(M_6^{Pl}(\Z/8))$ with
$\rho(\Z/8\Z)=<diag(1,\xi_8,\xi_8^4)>$, where $\xi_8$ is a $8$-th
primitive root of unity in $K${. Moreover,} such loci has
$X^5+Y^4Z+XZ^4+\beta X^3Z^2=0$ {as} a ``normal form'' with {a
parameter $\beta$ {that takes values in $K$}} such
that $\beta\neq\pm 2$ {(to ensure the non-singularity)}. Therefore,
we can associate to $\delta$ a {fixed} plane non-singular model of
the form $X^5+Y^4Z+XZ^4+\beta_{\delta} X^3Z^2=0$ for {some}
$\beta_{\delta}\in K$ ({there is no guarantee that} $\beta_{\delta}$
is unique in $K$).

\par Now, let us compute all {the} non-singular plane models of the form
$X^5+Y^4Z+XZ^4+\beta X^3Z^2=0$ that can be associated to the fixed
curve $\delta$. {These models are} obtained by a change of the
variables $P\in PGL_3(K)$ such that
$P^{-1}<(diag(1,\xi_8,\xi_8^4)>P=<diag(1,\xi_8,\xi_8^4)>$, and the
new model has a similar {defining equation of the} form
$X^5+Y^4Z+XZ^4+\beta'X^3Z^2=0$.
\par Without any loss of generality, we can suppose that
$P^{-1}diag(1,\xi_8,\xi_8^4)P=diag(1,\xi_8,\xi_8^4)${. Hence,} in
order to have the same eigenvalues which are pairwise distinct, we
may assume that $P$ is a diagonal matrix, say
$P=diag(1,\lambda_2,\lambda_3)$. Therefore, we get an equation of
the form: $X^5+\lambda_2^4\lambda_3 Y^4Z+\lambda_3^4
XZ^4+\beta_{\delta}\lambda_3^2X^3Z^2=0$. From which we must have
$\lambda_2^4\lambda_3=\lambda_3^4=1$, thus $\lambda_3^2$ is 1 or -1.
{Consequently,} we obtain a bijection map
$$\varphi:M_6^{Pl}(\Z/8\Z)\rightarrow \mathbb{A}^1(K)\setminus\{-2,2\}/\sim$$
$$\delta\mapsto [\beta_{\delta}]=\{\beta_{\delta},-\beta_{\delta}\}$$
where $a\sim b\Leftrightarrow b=a\ or\ a=-b$. {Furthermore, by our
work in} \cite{BaBaaut}, we know that $X^5+Y^4Z+XZ^4+\beta X^3Z^2=0$
has a bigger automorphism group than $\Z/8\Z$ if and only if
$\beta=0${. T}herefore, we have a bijection map
$$\tilde{\varphi}:\widetilde{M_6^{Pl}(\Z/8\Z)}\rightarrow \mathbb{A}^1(K)\setminus\{-2,0,2\}/\sim$$
$$\delta\mapsto [\beta_{\delta}]=\{\beta_{\delta},-\beta_{\delta}\}.$$
The above sets {are irreducible when} $K$ is the complex field.

{On the other hand}, if we consider the Galois cyclic cover of
degree 8 {that is} given by the action of the automorphism of order
8 on $X^5+Y^4Z+XZ^4+\beta X^3Z^2=0$, we obtain that it ramifies at
the points {$[0:1:0]$ and $[0:0:1]$} with ramification index 8, as
well as the four points {$[1:0:h],$} where $1+h^4+\beta h^2=0$ with
ramification index 2, if $\beta\neq \pm 2$. That is,
$M_6^{Pl}(\Z/8\Z)$ is inside the locus of curves {of the moduli
space} $M_6$ {that} have a cyclic Galois subcover of degree 8 to a
genus zero curve, and {also ramify} at {six} points ({two of them
are} with ramification index 8, and the other four points are with
ramification index 4).

\section{Preliminaries on automorphism on plane curves}
{Consider a curve} $\delta\in M_g^{Pl}$ {whose} $Aut(\delta)$ {is}
non-trivial{, and $C$ is a fixed non-singular, degree $d$ plane
model of $\delta$.} By an abuse of notation (once and for all), we
denote also by $C$ a non-singular projective plane curve {in
$\mathbb{P}^2$}. Then, $Aut(C)$ is a finite subgroup of $PGL_3(K)$,
and it satisfies one of the following situations (for more details,
see Mitchell \cite{Mit}):

\begin{enumerate}
\item fixes a point $Q$ and a line $L$ with $Q\notin L$ in
$PGL_3(K)$,
\item fixes a triangle {(i.e. a set of three non-concurrent lines)},
\item $Aut(C)$ is conjugate to a representation inside $PGL_3(K)$ of
one of the finite primitive groups namely, the Klein group
$PSL(2,7)$, the icosahedral group $A_5$, the alternating group
$A_6$, the Hessian group $Hess_{216}$ or to one of its subgroups $Hess_{72}$ or $Hess_{36}$.
\end{enumerate}

{Recall that the exponent of a non-zero monomial $cX^iY^jZ^k$ is
defined to be} $max\{i,j,k\}$. For a homogeneous polynomial $F$, the
core of $F$ is defined as the sum of all terms of $F$ with the
greatest exponent. Let $C_0$ be a smooth plane curve, a pair
$(C,\underline{G})$ with $\underline{G}\leq Aut(C)$ is said to be a
descendant of $C_0$ if $C$ is defined by a homogeneous polynomial
whose core is a defining polynomial of $C_0$ and $\underline{G}$
acts on $C_0$ under a suitable {change of the} coordinate system.

\begin{thm}[Harui] (see \cite{Harui} $\S 2$) \label{teo1}
Let $\underline{G}$ be a subgroup of $Aut(C)$ {where $C$ is a
non-singular plane curve of degree $d\geq4$}. Then $\underline{G}$
satisfies one of the following statements:
\begin{enumerate}
  \item $\underline{G}$ fixes a point on $C$ and then is cyclic.
  \item $\underline{G}$ fixes a point not lying on $C$ and it satisfies a short exact sequence of the form
  $$1\rightarrow N\rightarrow \underline{G}\rightarrow G'\rightarrow 1,$$
{where $N$ a cyclic group of order dividing $d$ and $G'$ (which is a
subgroup of $PGL_2(K)$) is conjugate to a cyclic group $\Z/m\Z$ of
order $m$ with $m\leq d-1$, a Dihedral group $D_{2m}$ of order $2m$
where $|N|=1$ or $m|(d-2)$, the alternating groups $A_4$, $A_5$ or
the symmetry group $S_4$.}
\item $\underline{G}$ is conjugate to a subgroup of $Aut(F_d)$, where $F_d$ is the Fermat curve $X^d+Y^d+Z^d$.
In particular, $|\underline{G}|\,|\,6d^2$ and $(C,\underline{G})$ is a descendant of $F_d$.
\item $\underline{G}$ is conjugate to a subgroup of $Aut(K_d)$,
where $K_d$ is the Klein curve curve $X^{d-1}Y+Y^{d-1}Z+Z^{d-1}X$.
Hence, $|H|\,|\,3(d^2-3d+3)$ and $(C,H)$ is a descendant of $K_d$.
\item $\underline{G}$ is conjugate to a finite primitive subgroup of $PGL_3(K)$ {that are mentioned above}.
\end{enumerate}
\end{thm}

\textbf{The Hessian group:} {The representations of the Hessian
group of order $216$ $Hess_{216}$ inside $PGL_3(K)$ forms a unique
set, up to conjugation (see Mitchell \cite{Mit} page $217$).} A
representation of $Hess_{216}$ in $PGL_3(K)$ is given by
$Hess_{216}=<S,T,U,V>$ {where}
$$S=\left(
\begin{array}{ccc}
1&0&0\\
0&\omega&0\\
0&0&\omega^2\\
\end{array}
\right),\ \ U=\left(
\begin{array}{ccc}
1&0&0\\
0&1&0\\
0&0&\omega\\
\end{array}
\right),\ V=\frac{1}{\omega-\omega^2}\left(
\begin{array}{ccc}
1&1&1\\
1&\omega&\omega^2\\
1&\omega^2&\omega\\
\end{array}
\right),\ \ T=\left(
\begin{array}{ccc}
0&1&0\\
0&0&1\\
1&0&0\\
\end{array}
\right).$$ {Here} $\omega$ is a primitive 3rd root of unity. Also,
we consider the primitive {Hessian} subgroups of order 36,
$Hess_{36}$ (one of them is $<S,T,V>$), and the primitive subgroup
of order 72, $Hess_{72}=<S,T,V,UVU^{-1}>$.

{For the above fixed representation, there} are exactly {three}
primitive subgroups of order 36 (see \cite{Grove}), which are {also}
normal in $Hess_{72}${. Moreover,} the Hessian subgroup $Hess_{72}$
is normal in $Hess_{216}$. {Furthermore, we recall,} by Grove in
\cite[\S 23,p.25]{Grove} and by Blichfeldt in \cite{Bli} (see also
\cite[\S1]{HaLe} for the statement of Blichfeldt's result of our
interest) that any representation of $Hess_{216}$ corresponds
geometrically to a certain subgroup fixing four triangles (having 18
elements), and the alternating group $A_4$ acting in such four
triangles. Moreover, any representation of the primitive subgroups
of order 36 or 72 is obtained by the group of 18 elements fixing the
four triangles {together} with certain permutations on the four
triangles ({equivalently, with} certain subgroups of $A_4$). {On the
other hand, it follows,} by Blichfeldt (see \cite[\S1, on type
(E),(F),(G)]{HaLe}), that such Hessian groups {are represented in
$PGL_3(K)$, up to conjugation, with respect to the representation
described above. Therefore,} any injective representation of
$Hess_{36}$ or $Hess_{72}$ in $PGL_3(K)$ extends to an injective
representation of $Hess_{216}$, and moreover the three different
subgroups of $Hess_{36}$ in any representation are conjugate to
$<S,T,V>$.
{Consequently, we conclude that} the representations of $Hess_{*}$
with $*\in\{36,72,216\}$ inside $PGL_3(K)$ form a unique set, up to
conjugation.

\begin{rem} \label{rem1}
  In particular, for the Hessian groups $Hess_{216}$, $Hess_{72}$
and $Hess_{36}$, the locus $\widetilde{M_g^{Pl}(Hess_*)}${, where
$*\in\{36,72,216\}$} is ES-Irreducible as long as {it} is not
empty{, since} the set $A_{Hess_*}$ is trivial ({we follow the same
notations} of \S2).
\end{rem}

{Our} interest in investigating {whether} the locus
$\widetilde{M_g^{Pl}(G)}$ is ES-irreducible or not, and the
classical result of Klein {concerning} the uniqueness (up to
conjugation) {of the} finite subgroups inside $PGL_2(K)$, {motivate
us to ask} the following question in group theory.
\begin{preg} {Is it true that there exists $\underline{G}$, a non-cyclic finite subgroup of $PGL_3(K)$, such that the set
$A_{\underline{G}}$ has at least two elements?}
\end{preg}

\section{Cyclic groups in smooth plane curves of degree 5 and $\widetilde{M_6^{Pl}(\Z/m\Z)}$.}
{We} study non-singular plane curves $C:\,F(X;Y;Z)=0$ of degree
$d\geq4$ such that $Aut(C)$ is non-trivial,  up to $K$-isomorphism{.
T}hat is, two of them are $K$-isomorphic if one transforms to the
other {through} a change of variables $P\in PGL_3(K)$, and we denote
by $C_P$ the plane curve $F(P(X;Y;Z))=0$.

By a change of variables, we can suppose that the cyclic group of order $m$ acting on a smooth plane curve of degree $5$ is
given in $PGL_3(K)$ by a diagonal matrix $diag(1;\xi_m^a;\xi_m^b)$,
where $\xi_m$ is an $m$-th primitive root of unity, and $0\leq a<
b<m$ are positive integers. We call this element by Type $m,(a,b)$.

Following the same proof of \cite[\S 6.5]{Dolgachev} (or see
\cite{BaBacyc}, for a general treatment with an algorithm of
computation for any degree $d$), we obtain  a ``normal form''
associated to type $m,(a,b)$ corresponding to the loci
$\rho(M_6^{Pl}(\Z/m\Z))$ with
$\rho(\Z/m)=<diag(1;\xi_m^a;\xi_m^b)>$:
\begin{center}
\begin{tabular}{|c|c|}
  \hline
  Type: $m, (a,\,b)$ & $F_{m,(a,b)}(X;Y;Z)$ \\\hline\hline
  $20,(4,5)$& $X^5+Y^5+ XZ^4$ \\\hline
  $16,(1,12)$& $X^5+Y^4Z+ XZ^4$\\\hline
  $15,(1,11)$& $X^5+Y^4Z+ YZ^4$    \\\hline
$13,(1,10)$& $X^4Y+Y^4Z+ Z^4X$    \\\hline
 $10,(2,5)$& $X^5+Y^5+ XZ^4+\beta_{2,0}X^3Z^2$ \\\hline
  $8,(1,4)$& $X^5+Y^4Z+ XZ^4+\beta_{2,0}X^3Z^2$ \\\hline
 $5,(1,2)$& $X^5+Y^5+Z^5+\beta_{3,1}X^2YZ^2+\beta_{4,3}XY^3Z$    \\\hline
  $5,(0,1)$& $Z^5+L_{5,Z}$    \\\hline

  $4,(1,3)$& $X^5+X\big(Z^4+ Y^4+\beta_{4,2}Y^2Z^2\big)+\beta_{2,1}X^3YZ$\\\hline
 $4,(1,2)$& $X^5+X\big(Z^4+ Y^4\big)+\beta_{2,0}X^3Z^2+\beta_{3,2}X^2Y^2Z+\beta_{5,2}Y^2Z^3$\\\hline
  $4,(0,1)$& $Z^4L_{1,Z}+L_{5,Z}$    \\\hline
  $3,(1,2)$& $X^5+Y^4Z+ YZ^4+\beta_{2,1}X^3YZ+X^2\big(\beta_{3,0}Z^3+\beta_{3,3}Y^3\big)+\beta_{4,2}XY^2Z^2$ \\\hline
 $2,(0,1)$& $Z^4L_{1,Z}+Z^2L_{3,Z}+L_{5,Z}$    \\\hline
  \end{tabular}
\end{center}
where $L_{i,U}$ {is} a homogeneous polynomial of degree
$i$ that does not contain the variable $U$ with parameters {as} the
coefficients {of} the monomials, and $\beta_{i,j}$ are parameters
{that assume} values in $K$. It remains to introduce the algebraic
restrictions that should be imposed on the parameters $\beta_{i,j}$
so that the defining equation $F_{m,(a,b)}(X;Y;Z)=0$ is
non-singular{. T}his will be omitted.

By the above table, we find that the locus $M_6^{Pl}(\Z/m\Z)$ is not
empty only for the values $m$ which are included in the previous
list. Moreover, we have
$M_6^{Pl}(\Z/m\Z)=\rho(M_6^{Pl}(\Z/m\Z))$ for $m\neq 4,5$, where $\rho$ is obtained
such that $\rho(\Z/m\Z)=<diag(1,\xi_m^a,\xi_m^b)>$. Thus, the
corresponding loci $\widetilde{M_6^{Pl}(\Z/m\Z)}$, where $m\neq4,5$,
are ES-Irreducible provided that they are non-empty.

Now, we consider the remaining cases of the loci
$\widetilde{M_6^{Pl}(\Z/m\Z)}$ with $m=4$ or $5$:

Obviously, the plane model of type $5, (1,2)$ always have a bigger
automorphism group by permuting $X$ and $Z$. Therefore, there is at
most one ``normal form'' that defines curves of degree $5$ whose
full automorphism group is isomorphic to $\Z/5\Z$ (observe that the
number of
  the conjugacy classes of representations of $\Z/5\Z$ in $PGL_3(K)$ is
  three). In particular, $\widetilde{M_6^{Pl}(\Z/5\Z)}$ is
ES-Irreducible if it is non-empty. More precisely,
$\widetilde{M_6^{Pl}(\Z/5\Z)}=\rho(M_6^{Pl}(\Z/5\Z))$ {with}
$\rho(\Z/5\Z)=<diag(1,1,\xi_5)>$ in this case.

On the other hand, for the cyclic groups of order 4, we have: Type
$4, (1,3)$ is not irreducible, since it is of the form $X\cdot
G(X;Y;Z)$. Hence, it is singular, and will be out of the scope of
this note. Therefore, we have
$M_6^{Pl}(\Z/4\Z)=\rho_1(M_6^{Pl}(\Z/4\Z))\cup
  \rho_2(M_6^{Pl}(\Z/4\Z))$, where $\rho_1$ corresponds to Type
  $4, (0,1)$ and $\rho_2$ to Type $4, (1,2)$.
%
\subsection{On type $4,(0,1)$} Consider the non-singular plane curve defined
by the equation $$\tilde{C}:\,\,X^5+Y^5+Z^4X+\beta X^3Y^2{=0},$$
where $\beta\neq0.$ This curve admits an automorphism of order $4$
namely, $\sigma:=[X;Y;\xi_4Z]$ that fixes pointwise the line $Z=0$
(its axis) and the point $[0:0:1]$ off this line (its center). We
call the elements of $PGL_3(K)$ that fix similar geometric
constructions, homologies (for the element
$diag(1;\xi_m^a;\xi_m^b)\in PGL_3(K)$ with $0\leq a<b<m$, is an
homology when $a=0$). It follows, by {$\S 5$ in \cite{Mit},} that
$Aut(\tilde{C})$ should fix a point, a line or a triangle.

If $Aut(\tilde{C})$ fixes a triangle and neither a line nor a point
is leaved invariant then, $\tilde{C}$ is a descendant of the Fermat
curve $F_5$ or the Klein curve $K_5$ (Harui \cite{Harui}, $\S 5$).
But this is impossible, because $4\nmid|Aut(F_5)|(=150)$, and
$4\nmid|Aut(K_5)|(=39)$. Therefore, $Aut(\tilde{C})$ should fix a
line and a point off that line.

Now, the point {$[0:0:1]$} is an inner Galois point of $\tilde{C}$,
by Lemma $3.7$ in \cite{Harui}. Also, it is unique, by Yoshihara
\cite{Yoshihara}, $\S 2$,
Theorem $4$. 
Therefore, this point must be fixed by $Aut(\tilde{C})$. Moreover,
the axis $Z=0$ is also leaved invariant by Mitchell \cite{Mit}, \S
4. In particular, $Aut(\tilde{C})$ is cyclic by Lemma 11.44 in
\cite{Book}, and automorphisms of $\tilde{C}$ are all
diagonal of the form $[X;vY;tZ]$. This in turns implies that
$v^5=v^2=t^4=1$. Hence, $v=1$ and $t$ is a $4$-th root of unity.
This shows that $Aut(\tilde{C})$ is cyclic of order $4$.

Therefore, with the above argument we conclude the following result.
\begin{prop} The locus set $\widetilde{\rho_1(M_6^{Pl}(\Z/4\Z))}$ is
non-empty.
\end{prop}

\subsection{On type $4, (1,2)$} Consider the non-singular plane curve
defined by the equation $$\tilde{\tilde{C}}:\,\,X^5+X(Z^4+Y^4)+\beta
Y^2Z^3{=0},$$ where $\beta\neq0$. This curve admits a cyclic
subgroup of automorphisms generated by $\tau:=[X;\xi_4Y;\xi_4^2Z]$.
For the same reason as above (i.e $4\nmid|Aut(K_5)|, |Aut(F_5)|$),
$\tilde{\tilde{C}}$ is not a descendant of the Fermat curve $F_5$ or
the Klein curve $K_5$. Moreover, $Aut(\tilde{\tilde{C}})$ is not
conjugate to an icosahedral group $A_5$ (no elements of order 4),
the Klein group $PSL(2,7)$, the Hessian group $Hess_{216}$ or the
alternating group $A_6$ (since by \cite{Harui}, Theorem $2.3$,
$|Aut(\tilde{\tilde{C}})|\leq150$).

Now, we claim to prove that $Aut(\tilde{\tilde{C}})$ is also not
conjugate to any of the Hessian subgroups namely, $Hess_{36}$ or
$Hess_{72}$, and therefore it should fix a line and a point off that
line:
Let $C$ be a non-singular plane curve of degree $5$ such that
$Aut(C)$ is conjugate, through $P\in PGL_3(K)$, to $Hess_{*}$ with
$*\in\{36,72,216\}$. Then $Aut(C_{P})$ is given by the usual
presentation inside $PGL_3(K)$ of the above Hessian groups. In
particular, $Aut(C_{P})$ always has the following five elements:
$[Z;Y;X]$, $[X;Z;Y]$, $[Y;X;Z]$, $[Y;Z;X]$ and $[X;\omega Y;\omega^2
Z]$, where $\omega$ is a primitive $3$-rd root of unity. Because
$C_{P}$ is invariant by $[Z;Y;X]$, $[X;Z;Y]$, $[Y;X;Z]$ and
$[Y;Z;X]$, then $C_{P}$ must be of the form:
$u(X^5+Y^5+Z^5)+a(X^4Z+X^4Y+Y^4X+Y^4Z+Z^4X+Z^4Y)+G(X;Y;Z)$, where
$u,a\in K$, and $G(X;Y;Z)$ is a homogenous polynomial of degree at
most three in each variable. Now, imposing {the condition}
$[X;\omega Y;\omega^2 Z]\in Aut(C_{P})$, we obtain that $u=0$ and
$a=0$, a contradiction to non-singularity. Therefore, there is no
non-singular, degree 5 plane curve whose automorphism group is
conjugate to one of the Hessian groups. This proves our claim.

It follows, by the previous discussion, that
$Aut(\tilde{\tilde{C}})$ should fix a line and a point off that
line. Moreover, $\tau\in Aut(\tilde{\tilde{C}})$ is of the form
$diag(1;a;b)$ such that $1,a,b$ (resp. $1,a^3,b^3$) are pairwise
distinct then, automorphisms of $\tilde{\tilde{C}}$ are of the forms
$\tau_1:=[X;vY+wZ;sY+tZ],\,\tau_2:=[vX+wZ;Y;sX+tZ]$ or
$\tau_3:=[vX+wY;sX+tY;Z]$ (because the fixed point is one of the
reference points $[1:0:0],\,[0:1:0]$ or $[0:0:1]$, and the fixed
line is one of the reference lines $X=0,\,Y=0$ or $Z=0$).
\par If $\tau_1\in Aut(\tilde{\tilde{C}})$ then $s=0=w$ (Coefficient of $Y^5$ and $Z^5$), and we have the same conclusion, if $\tau_2$ (resp. $\tau_3$) $\in Aut(\tilde{\tilde{C}})$ from the coefficients of $X^3Y^2$ and $Y^4Z$ (resp. $Z^3X$ and $YZ^4$). Hence, automorphisms of $\tilde{\tilde{C}}$ are all diagonal of the form $[X;vY;sZ]$. Moreover, $v^4=s^4=v^2s^3=1$, hence $v=\xi_4^r,\,s=\xi_4^{r'}$ with $(r,r')\in\{(0,0),\,(2,0),\,(1,2),\,(3,2)\}$. That is, $Aut(\tilde{\tilde{C}})$ is cyclic of
order 4.

\noindent Consequently, the following results follow.
\begin{prop} The locus set $\widetilde{\rho_2(M_6^{Pl}(\Z/4\Z))}$ is
non-empty.
\end{prop}

\begin{cor} The locus set $\widetilde{M_6^{Pl}(\Z/m\Z)}$ is
ES-Irreducible if and only if $m\neq4$. If $m=4$ then
$\widetilde{M_6^{Pl}(\Z/m\Z)}$ has exactly two irreducible equation
components, and hence the number of its irreducible components is at
least two.
\end{cor}

\begin{rem}\label{rem2}
Observe that for any element of $\rho_1(M_6^{Pl}(\Z/4\Z))$, the
Galois cover of degree 4 corresponding to
$${C}_{{\rho_1}}:=Z^4L_{1,Z}+L_{5,Z}=0\rightarrow {C}_{\rho_1}/<[X;Y;\xi_4Z]>$$
is ramified exactly at six points with ramification index 4. Indeed,
the fixed points of $\sigma^i$ for $i=1,2,3,4$ in $\mathbb{P}^2(K)$
are all the same set where $\sigma=diag(1,1,\xi_4)${. T}herefore, we
only need to consider the ramification points of $\sigma${. I}n
particular, the ramification index is always 4. Now, by the Hurwitz
formula, we {have} $10=4(2 g_0-2)+3k$ where $g_0$ is the genus of
${C}_{\rho_1}/<[X,Y,\xi_4 Z]>${. Hence, $g_0=0$ and $k=6$}. On the
other hand, for any element of $\rho_2(M_6^{Pl}(\Z/4\Z))$, the
Galois cover
$${{C}}_{{\rho_2}}:=X^5+X(Z^4+
Y^4)+\beta_{2,0}X^3Z^2+\beta_{3,2}X^2Y^2Z+\beta_{5,2} Y^2
Z^3=0\rightarrow{C}_{\rho_2}/<[X;\xi_4 Y;\xi_4^2 Z]$$ is ramified at
the points $[0:1:0],[0:0:1]$ with ramification index 4, and at the 4
points namely, $[1:0:h]$ where $1+h^4+\beta_{2,0}h^2=0$ with
ramification index 2 provided that $\beta_{2,0}\neq \pm 2$. {We
exclude the situation $\beta_{2,0}=\pm 2$ so that the defining
equation is non-singular
and geometrically irreducible}.
\end{rem}

\begin{rem} Given $G$, a non-trivial finite group, such that $\widetilde{M_6^{Pl}(G)}$ is
non-empty. By a tedious work, one can show that
$\widetilde{M_6^{Pl}(G)}$ is ES-Irreducible, except for the case
$G\cong\Z/4\Z$ (for more details, we refer to \cite{BaBaaut}).
\end{rem}

\begin{thm} Let $d\geq 5$ be an odd integer, and consider $g=(d-1)(d-2)/2$ as usual. Then $\widetilde{M_g^{Pl}(\Z/(d-1)\Z)}$ is not
ES-Irreducible, and it has at least two irreducible components.
\end{thm}
\begin{proof}
The above argument for concrete curves of Type $4, (0,1)$ and Type
$4, (1,2)$ is valid for any odd degree $d\geq5$, and the proof is
quite similar. In other words, let $\tilde{C}$ and
$\tilde{\tilde{C}}$ be the non-singular plane curves of types $d-1,
(0,1)$ and $d-1, (1,2)$ defined by the equations
$X^d+Y^d+Z^{d-1}X+\beta X^{d-2}Y^2=0$, and
$X^d+X(Z^{d-1}+Y^{d-1})+\beta Y^2Z^{d-2}=0$ {respectively}, where
$\beta\neq0$. Then, $Aut(\tilde{C})$ and $Aut(\tilde{\tilde{C}})$
are non-conjugate cyclic groups of order $d-1$, and are generated by
$[X;Y;\xi_{d-1}Z]$ and $[X;\xi_{d-1}Y;\xi_{d-1}^2Z]$ respectively.
Therefore, they belong to two different
$[\rho]'s$.\\\\
\noindent\textbf{On type $d-1, (0,1)$:} With a homology of order
$d-1\geq 4$ inside $Aut(\tilde{C})$, we conclude that
$Aut(\tilde{C})$ fixes a point, a line or a triangle (See
\cite{Mit}, \S 5). Furthermore, the center {$[0:0:1]$} of this
homology is an inner Galois point, by Lemma $3.7$ in \cite{Harui}.
Also, it is unique, by Theorem $4$ in \cite{Yoshihara}. Therefore,
it should be fixed by $Aut(\tilde{C})$, and also the axis $Z=0$ is
leaved invariant, by Theorem 4 in \cite{Mit}. Hence,
$Aut(\tilde{C})$ is cyclic, by Lemma $11.44$ in \cite{Book}, and
automorphisms of $\tilde{C}$ are of the form $diag(1;v;t)$ such that
$v^d=t^{d-1}=v=1$. That is, $|Aut(\tilde{C})|=d-1$.
\\\\
\noindent\textbf{On type $d-1, (1,2)$:} First, we prove that
$Aut(\tilde{\tilde{C}})$ fixes a line and a point off this line:
We consider the case $d\geq7$ (For $d=5$, we refer to the previous
results). The alternating group $A_6$ has no elements of order
$d-1\geq6$. The Klein group $PSL(2,7)$, which is the only simple
group of order $168$, has no elements of order $\geq8$, and also
there are no elements of order $6$ inside (for more details, we
refer to \cite{Fano}). Therefore, the primitive groups $A_5, A_6$,
and $PSL(2,7)$ do not appear as the full automorphism group.
Moreover, elements inside the Hessian group $Hess_{216}\cong
SmallGroup(216,153)$ have orders $1,2,3,4$ and $6$. Then $Hess_{*}$
with $*\in\{36,72,216\}$ do not appear as the full automorphism
group, except possibly for $d=7$. On the other hand, $d-1\nmid
3(d^2-3d+3)$ hence $\tilde{\tilde{C}}$ is not a descendant of the
Klein curve $K_d$. Furthermore, $\tilde{\tilde{C}}$ is not a
descendant of the Fermat curve $F_d$, because $d-1\nmid 6d^2$
(except for $d=7$).

Finally, it remains to deal with the case $d=7$ for the Hessian
groups or for being a Fermat's descendant. By the same line of
argument as for the claim of Type $4, (1,2)$, we can show that non
of the Hessian groups could appear for a non-singular, degree 7,
plane curve. Also, the automorphisms of the Fermat curve $F_7$ are
of the forms
$[X;\xi_{7}^aY;\xi_{7}^bZ],\,[\xi_{7}^bZ;\xi_{7}^aY;X],\,[X;\xi_{7}^bZ;\xi_{7}^aY],\,[\xi_{7}^aY;X;\xi_{7}^bZ],\,[\xi_{7}^aY;\xi_{7}^bZ;X],\,[\xi_{7}^bZ;X;\xi_{7}^aY].$
One can easily verify that non of them has order $6$. Consequently,
we exclude the possibility of being a Fermat's descendant.
\par Now, the full automorphism group should fix a line and a point off this line. Thus automorphisms of $\tilde{\tilde{C}}$ have the forms $[X;vY+wZ;sY+tZ],\,[vX+wZ;Y;sX+tZ]$ or $[vX+wY;sX+tY;Z]$, since $[X;\xi_{d-1}Y;\xi_{d-1}^2Z]\in Aut(\tilde{\tilde{C}})$.

If $[X;vY+wZ;sY+tZ]\in Aut(\tilde{\tilde{C}})$ then $s=0=w$
(Coefficient of $Y^d$ and $Z^d$), and the same conclusion follows if
$[vX+wZ;Y;sX+tZ]$ (resp. $[vX+wY;sX+tY;Z]$) $\in
Aut(\tilde{\tilde{C}})$ from the coefficients of $X^{d-2}Y^2$ and
$Y^{d-1}Z$ (resp. $Z^{d-2}X^2$ and $YZ^{d-1}$). Hence, automorphisms
of $\tilde{\tilde{C}}$ are all diagonal of the form $diag(1;v;s).$
Moreover, $v^{d-1}=s^{d-1}=v^2s^{d-2}=1$ that is, $v=\xi_{d-1}^r$
and $s=\xi_{d-1}^{r'}$ such that $d-1|2r-r'$. Therefore,
automorphisms of $\tilde{\tilde{C}}$ are
$[X;\xi_{d-1}^rY;\xi_{d-1}^{2r}Z]$ with $r\in\{0,1,...,d-2\}$. Hence,
$Aut(\tilde{\tilde{C}})$ is cyclic of order $d-1$, which was to be
shown.

\end{proof}

\section{On the locus $\widetilde{M_{10}^{Pl}(\Z/3\Z)}$.}
By a similar argument as the 
degree 5 case, we obtain the following ``normal forms'' for
$\rho(M_{10}^{Pl}(\Z/3\Z))$, (see the full table {for} degree 6 in
\cite{BaBacyc}):

\begin{center}

\begin{tabular}{|c|c|}
  \hline
  Type: $m, (a,\,b)$ & $F_{m,(a,b)}(X;Y;Z)$ \\\hline\hline
$3,(0,1)$& $Z^6+Z^3 L_{3,Z}+L_{6,Z}$\\ \hline $3,(1,2)$ &
$X^5Y+Y^5Z+Z^5X+\mu_1 Z^2X^4+\mu_2 X^2 Y^4+\mu_3 Y^2Z^4+\alpha_1
X^3Y^2Z+\alpha_2 XY^3Z^2+\alpha_3 X^2YZ^3$\\ \hline
\end{tabular}

\end{center}
where $\mu_i,\alpha_i$ {are} parameters {that take} values in $K$
{so that the associated models of the respective loci
$\rho(M_{10}^{Pl}(\Z/3\Z))$ are non-singular}.

\subsection{On type $3,(1,2)$}

\begin{prop}\label{nonprimitive}
Let $\delta\in M_{10}^{Pl}(\Z/3\Z)$ such that $\delta$ admits a
non-singular plane model $\tilde{C}$ of the form
$$X^5Y+Y^5Z+Z^5X+\mu_1 Z^2X^4+\mu_2 X^2Y^4+\mu_3 Y^2Z^4+\alpha_1 X^3Y^2Z+\alpha_2 XY^3Z^2+\alpha_3 X^2YZ^3=0.$$
Then, $Aut(\tilde{C})$ either fixes a line and a point off that line
or it
fixes a triangle. 
\end{prop}

\begin{proof}
It suffices to show that $Aut(\tilde{C})$ is not conjugate to any of
the finite primitive groups inside $PGL_3(K)$ namely, the Klein
group $PSL(2,7)$, the icosahedral group $A_5$, the alternating group
$A_6$, the Hessian group $Hess_{216}$ or to any of its subgroups
$Hess_{72}$ or $Hess_{36}$, and the result follows by Mitchell \cite{Mit}.

Let $\tau\in Aut(\tilde{C})$ be an element of order $2$ such that
$\tau\sigma\tau=\sigma^{-1}$, where $\sigma:=[X;\omega Y;\omega^2Z]$
then $\tau$ has one of the forms $[X;\beta Z;\beta^{-1}Y],\,[\beta
Y;\beta^{-1}X;Z]$ or $[\beta Z;Y;\beta^{-1}X]$. But non of these
transformations retains $\tilde{C}$, hence $Aut(\tilde{C})$ does not
contain an $S_3$ as a subgroup. Consequently, $Aut(\tilde{C})$ is
not conjugate to $A_5$ or $A_6$. Moreover, it is well known that
$PSL(2,7)$ contains an octahedral group of order $24$ (but not an
isocahedral group of order $60$), and since all elements of order 3
in $PSL(2,7)$ are conjugate (for more details, we refer to
\cite{Fano}). Then, by the same argument as before, we conclude that
$Aut(\tilde{C})$ is not conjugate to $PSL(2,7)$. Lastly, assume that
$Aut(\tilde{C})$ is conjugate, through a transformation $P$, to one
of the Hessian groups say, $Hess_{*}$. Then, we can consider
$P^{-1}SP=\lambda S$, because we did not fix the plane model for a
curve whose automorphism group is $Hess_{*}$. In particular, $P$
should be of the form $[Y;\gamma Z;\beta X],\,[Z;\gamma X;\beta Y]$
or $[X;\gamma Y;\beta Z]$, but non of them transform $\tilde{C}$ to
$\tilde{C}_P$ with $\{[X;Z;Y],\,[Y;X;Z],\,[Z;Y;X]\}\subseteq
Aut(\tilde{C}_P)$. Therefore, $Aut(\tilde{C})$ is not conjugate to
any of the Hessian groups, and we have done.
\end{proof}


{{\bf{Notations.}}\,\,Let $\Gamma:=\{(\beta'_1,\beta'_2,\beta'_3)\in
K^{*}\times K^{*}\times K^{*}:\,\,{
\Upsilon_1=1,\,\Upsilon_2=\Upsilon_3=\omega^2\Upsilon_4}\},$ where
\begin{eqnarray*}
{\Upsilon_1\left(\beta_1^{'},\beta_2^{'},\beta_3^{'}\right)}&:=&\beta_3^{'}\beta_2^{'5}+\left(\beta_1^{'}\beta_3^{'3}+1\right)\beta_2^{'}+\beta_3^{'5},\\
{\Upsilon_2\left(\beta_1^{'},\beta_2^{'},\beta_3^{'}\right)}&:=&
\lambda ^{'2}\left(\left(5 \lambda ^{'3}+1\right) \beta_3^{'} \beta
_2^{'5}+\left(5 \lambda ^{'6}+\lambda^{'3}+\left(2 \lambda
^{'6}+\lambda ^{'3}+3\right)
   \beta _1^{'} \beta _3^{'3}\right) \beta_2^{'}+\left(\lambda ^{'6}+5\right) \beta_3^{'5}\right),\\
{\Upsilon_3\left(\beta_1^{'},\beta_2^{'},\beta_3^{'}\right)}&:=&\lambda
^{'5} \left(\left(\lambda ^{'6}+5\right) \beta_3^{'}
\beta_2^{'5}+\left(5 \lambda ^{'3}+\left(3
   \lambda ^{'6}+2 \lambda^{'3}+1\right) \beta_1^{'}
   \beta _3^{'3}+1\right) \beta _2^{'}+\lambda^{'3} \left(5
   \lambda^{'3}+1\right) \beta_3^{'5}\right),\\
{\Upsilon_4\left(\beta_1^{'},\beta_2^{'},\beta_3^{'}\right)}&:=&\lambda^{'}
\left(\lambda^{'4} \left(5 \lambda^{'3}+1\right) \beta_3^{'}
\beta_2^{'5}+\lambda^{'}\left(\lambda^{'6}+\left(\lambda^{'6}+3
\lambda^{'3}+2\right) \beta_1^{'} \beta_3^{'3}+5\right)
\beta_2^{'}+\lambda^{'}\left(5 \lambda^{'3}+1\right)
\beta_3^{'5}\right),
 \end{eqnarray*}
{and} {$\lambda^{'3}=\xi_6^2=\omega$.} Also, we define
$\Gamma_1$ to be the set of all values that appear in the first
coordinate of elements of $\Gamma$, which ({by a
computation})
is a finite subset of $K^*$. 
Now, we state and prove the main result for this section:
\begin{thm}
Consider an element $\delta\in M_{10}^{Pl}(\Z/3\Z)$ that has a
non-singular plane model $\tilde{C}$ of the form
$\tilde{C}:\,\,X^5Y+Y^5Z+Z^5X+\alpha_3 X^2YZ^3=0$ with $\alpha_3\neq0${, and assume for simplicity that $\alpha_3\notin \Gamma_1$}. 
The full automorphism group of such $\delta$ is cyclic of order $3$,
and is generated by the transformation $\sigma:\,(x;y;z)\mapsto
(x;\omega y;\omega^2z)$.
\end{thm}

\begin{proof}
It follows, by Proposition \ref{nonprimitive}, that $Aut(\tilde{C})$
either fixes a line and a point off that line or it fixes a
triangle. We treat each of these two cases.
\begin{enumerate}
  \item If $Aut(\tilde{C})$ fixes a line $L$ and a point $P$ off this line, then $L$ must be one of the reference lines $B=0$, where $B\in\{X,Y,Z\}$, and $P$ is one of the reference points namely, $[1:0:0],\,[0:1:0]$ or $[0:0:1]$ (being $\sigma\in Aut(\tilde{C})$). Consequently, $Aut(\tilde{C})$ is cyclic, since all the reference points lie on $\tilde{C}$. Also, automorphisms of $\tilde{C}$ are of the forms $$\tau_1:=[X;vY+wZ;sY+tZ],\,\tau_2:=[vX+wZ;Y;sX+tZ]\,\,\,\, or\,\,\, \tau_3:=[vX+wY;sX+tY;Z]$$
For $\tau_1$ to be in $Aut(\tilde{C})$, we must have $w=0=s$
(coefficients of $X^5Z$ and $XY^5$), and similarly, for $\tau_2$
(resp. $\tau_3$) through  the coefficients of $Y^5X$ and $Z^6$
(resp. $YZ^5$ and $X^5Z$). That is, elements of $Aut(\tilde{C})$ are
all diagonal of the form $diag(1;v;t)$ such that $tv^4=1=t^3$ and
$t^5=v$. Thus, $t=\xi_{3}^a$ and $v=\xi_{3}^{2a}$, where $\xi_{3}$
is a primitive $3$-rd root of unity, and hence, $|Aut(C)|=3$.
  \item If $Aut(\tilde{C})$ fixes a triangle and there exist neither a line nor a point leaved invariant, then by Harui \cite{Harui}, $\tilde{C}$
   is a descendant of the Fermat curve $F_6:\,\, X^6+Y^6+Z^6$ or the Klein curve $K_6:\,\,X^5Y+Y^5Z+Z^5X$. Hence, $Aut(\tilde{C})$ is conjugate to a subgroup of $Aut(F_6)=<\,[\xi_6X;Y;Z],\,[X;\xi_6Y;Z],\,[Y;Z;X],\,[X;Z;Y]\,>$ or to a subgroup of $Aut(K_6)=<\,[Z;X;Y],\,\,[X;\xi_{21}Y;\xi_{21}^{-4}Z]\,>.$
      \begin{itemize}
        \item Suppose first that $Aut(\tilde{C})$ is conjugate (through $P$) to a subgroup of $Aut(F_6)$. Then, it suffices to assume that $P^{-1}SP\in\{S,\,[Y;Z;X],\,[Y;\xi_6Z;X],\,[Y;\xi_6^2Z;X]\}$, since any element of order 3 in $Aut(F_6)$, which is not a homology, is conjugate to one of those inside $Aut(F_6)$. Now, if $P^{-1}SP=S$ then $P\in PGL_3(K)$
        is of the form $[Y;\gamma Z;\beta X],\,\,[Z;\gamma X;\beta Y]$ or $[X;\gamma Y;\beta Z]$,
        but non of them transforms $\tilde{C}$ to $\tilde{C}_P$ with core $X^6+Y^6+Z^6$, a contradiction. Furthermore, if $P^{-1}SP=[Y;Z;X]$
         (resp. $=[Y;\xi_6Z;X]$ or $=[Y;\xi_6^2Z;X]$), then
            $P$ has the form $\left(
      \begin{array}{ccc}
      \lambda & 1 & \lambda^2 \\
      \omega\lambda\beta_2 & \beta_2 & \omega^2\lambda^2\beta_2 \\
      \omega^2\lambda\beta_3 & \beta_3 & \lambda^{2}\omega\beta_3 \\
       \end{array}
        \right)$, where $\lambda^3=1$ (resp. $\lambda^3=\xi_6$ or $\lambda^3=\xi_6^2$). We thus get $\tilde{C}_P$ of the form
         {$\Upsilon_1\left(\alpha_3,\beta_2,\beta_3\right)(\lambda^6\omega X^6+Y^6+\lambda^{12}\omega^2 Z^6)+...$
         . In particular, $\Upsilon_1\left(\alpha_3,\beta_2,\beta_3\right)=1$, $\lambda^3=\xi_6^2$ and $[Y;\xi_6^2Z;X]\in Aut(\tilde{C}_P)$. Consequently, $\tilde{C}_P$ should be of the form 
{$
   X^6+Y^6+Z^6+\left(\Upsilon_2\left(\alpha_3,\beta_2,\beta_3\right)X^5Y+\Upsilon_4\left(\alpha_3,\beta_2,\beta_3\right) Y^5Z+
   \Upsilon_3\left(\alpha_3,\beta_2,\beta_3\right)XZ^5\right)+...,
$
which must be reduced to the form
   $
   X^6+Y^6+Z^6+\Upsilon_2\left(\alpha_3,\beta_2,\beta_3\right)\left(X^5Y+\omega Y^5Z+XZ^5\right)+...,
   $ because $[Y;\xi_6^2Z;X]\in Aut(\tilde{C}_P)$. This could happen only if $\alpha_3\in \Gamma_1$,}
   which is not possible by the assumptions on $\alpha_3$
         }
        Therefore, $\tilde{C}$ is not a descendant of the Fermat curve $F_6$.
%
%
%
%
%
%

        \item Secondly, suppose that $\tilde{C}$ is a descendant of the Klein curve $K_6$. This should happen through a change of the variables $P\in PGL_3(K)$ such that $\tilde{C}_P:\,X^5Y+Y^5Z+Z^5X+lower\,\,\,terms.$ We claim to show that $P^{-1}SP=\lambda S$ for some $\lambda\in K^*$. Indeed, elements of order $3$ inside $Aut(K_6)$, which are not homologies, are $S,S^{-1},[\xi_{21}^aY;\xi_{21}^{-4a}Z;X]$ and $[\xi_{21}^{-4a}Z;X;\xi_{21}^aY]$, and it is enough to consider the situation $P^{-1}SP\in\{S,\,S^{-1},\,[\xi_{21}^aY;\xi_{21}^{-4a}Z;X],\,[\xi_{21}^{-4a}Z;X;\xi_{21}^aY]\}$ with $a=0,1,2$, because any other value is conjugate inside $Aut(K_6)$ to one of these transformations.
            \par If $P^{-1}SP=\lambda S^{-1}$ then $P$ fixes one of the variables and permutes the others. Hence, the resulting core is different from $X^5Y+Y^5Z+Z^5X$, a contradiction.
            \par If $P^{-1}SP=\lambda [\xi_{21}^aY;\xi_{21}^{-4a}Z;X]$ (resp. $[\xi_{21}^{-4a}Z;X;\xi_{21}^aY]$) then $P$ has the form
$$\left(
    \begin{array}{ccc}
      \lambda\xi_{21}^{-a} & 1 & \lambda^2\xi_{21}^{-a} \\
      \lambda\xi_{21}^{-a}\omega\beta_2 & \beta_2 & \lambda^2\xi_{21}^{-a}\omega^2\beta_2 \\
      \lambda\xi_{21}^{-a}\omega^2\beta_3& \beta_3 & \lambda^{2}\xi_{21}^{-a}\omega\beta_3 \\
    \end{array}
  \right)
\,\,(resp.\,\left(
    \begin{array}{ccc}
      \lambda^{2}\xi_{21}^{-18a} & 1 & \lambda\xi_{21}^{-a} \\
      \lambda^{2}\xi_{21}^{-18a}\omega^2\beta_2 & \beta_2 & \lambda\xi_{21}^{-a}\omega\beta_2 \\
      \lambda^{2}\xi_{21}^{-18a}\omega\beta_3 & \beta_3 & \lambda\xi_{21}^{-a}\omega^{2}\beta_3 \\
    \end{array}
  \right) )$$
where $\lambda^3=\xi_{21}^{-3a}$. For both transformations, we must
have $\beta _3 \beta _2^5+\left(\alpha _3 \beta
   _3^3+1\right) \beta _2+\beta _3^5=0$ so that $X^6,Y^6,Z^6$ do not appear. Therefore, by imposing the condition $X^5Z,XY^5$ and $YZ^5$ do not appear as well, we get $\alpha_3=0$, which is already excluded.
%
%
%
Consequently, $P^{-1}SP=\lambda S$, and we proved the claim. Now,
$P$ has one of the forms $[Y;\gamma Z;\beta X],\,\,[Z;\gamma X;\beta
Y]$ or $[X;\gamma Y;\beta Z]$. Therefore, $\tilde{C}_P$ is defined
by an equation of the form
$\lambda_0(X^5Y+Y^5Z+Z^5X)+\lambda_1G(X;Y;Z)$, where $G(X;Y;Z)$ {is
one of the monomials $X^2YZ^3,\,Y^2ZX^3,$ or $Z^2XY^3$}. In
particular, $[\mu_1Z;X;\mu_2Y]\notin Aut(\tilde{C}_P)$, and
$Aut(\tilde{C}_P)\preceq <\,\tau:=[X;\xi_{21}Y;\xi_{21}^{-4}Z]>$.
Moreover, $\tau^r\in Aut(\tilde{C}_P)$ if and only if $7|r$. Hence,
$Aut(\tilde{C})$ is cyclic of order 3.
      \end{itemize}

\end{enumerate}
This completes the proof.
\end{proof}

\subsection{On type $3,(0,1)$}
\begin{prop}\label{lem4}
If $\delta\in M_{10}^{Pl}(\Z/3\Z)$ has a non-singular plane model
$\tilde{\tilde{C}}$ of the form $Z^6+Z^3L_{3,Z}+L_{6,Z}=0$, then
$Aut(\tilde{\tilde{C}})$ is either conjugate to the Hessian group
$Hess_{216}$ or it leaves invariant a point, a line or a triangle.
\end{prop}

\begin{proof}
The result is an immediate consequence, since
$Aut(\tilde{\tilde{C}})$ contains a homology (i.e. leaves invariant
a line pointwise and a point off this line) of period $3$ namely,
$\sigma':=[X;Y;\omega Z]$, and $Hess_{216}$ is the only
multiplicative group that contains such homologies and does not
leave invariant a point, a line or a triangle (See Theorem $9$,
\cite{Mit}).
\end{proof}

Now, we can prove our main result for this section.
\begin{thm}
The automorphisms group of an element $\delta\in
M_{10}^{Pl}(\Z/3\Z)$ with a non-singular plane model
$\tilde{\tilde{C}}$ of the form $Z^6+X^5Y+XY^5+\alpha_3Z^3X^3=0$
such that $\alpha_3\neq0$ is cyclic of order $3$, and is generated
by the automorphism $\sigma':\,(x;y;z)\mapsto (x;y;\omega z)$.
\end{thm}

\begin{proof}
Suppose that $Aut(\tilde{\tilde{C}})$ is conjugate, through a
transformation $P$, to the Hessian group $Hess_{216}$. Then, we can
assume, without loss of generality, that $P^{-1}\sigma' P=\lambda
\sigma'$ for some $\lambda\in K^*$. Hence,
$P=[\alpha_1X+\alpha_2Y;\beta_1X+\beta_2Y;Z]$ and clearly,
$\{[Z;Y;X],[X;Z;Y]\}\nsubseteq Aut(\tilde{\tilde{C}}_P)$, a
contradiction. Therefore, we deduce, by Proposition \ref{lem4}, that
$Aut(\tilde{\tilde{C}})$ should fix a point, a line or a triangle.

In what follows, we treat each case.
\begin{enumerate}
  \item If $Aut(\tilde{\tilde{C}})$ fixes a line and a point off that line, and if $\tilde{\tilde{C}}$ admits a bigger non cyclic automorphism group,
  then $Aut(\tilde{\tilde{C}})$ satisfies a short exact sequence
  of the form $1\rightarrow{ \Z/3\Z}\rightarrow Aut(\tilde{\tilde{C}})\rightarrow G'\rightarrow 1$, where $G'$ is conjugate to {$\Z/m\Z$} ($m=2,3$ or $4$), $D_{2m}$ ($m=2$ or $4$), $A_4,S_4$ or $A_5$.
      \par If $G'$ is conjugate to ${\Z/3\Z},A_4,S_4$ or $A_5$, then there exists, by Sylow's theorem, a subgroup $H$ of automorphisms of  $\tilde{\tilde{C}}$ of order $9$. In particular, $H$ is conjugate to {$\Z/9\Z$ or $\Z/3\Z\times \Z/3\Z$}, but both cases do not occur. Indeed, if $H$ is conjugate to {$\Z/9\Z$} then $Aut(\tilde{\tilde{C}})$ has an element of order $9$, which is not possible because $9\nmid d-1,d,(d-1)^2,d(d-2),d(d-1),d^2-3d+3$ with $d=6$ (for more details, we refer to \cite{BaBacyc}). Moreover, if $H$ is conjugate to {$\Z/3\Z\times \Z/3\Z$} then there exists $\tau\in Aut(\tilde{\tilde{C}})$ of order $3$ such that $\tau\sigma'=\sigma'\tau$. Hence, $\tau=[vX+wY;sX+tY;Z]$, and comparing the coefficients of $Z^3Y^3$ and $X^6$ in $\tilde{\tilde{C}}_{\tau}$, we get $w=0=s$ and $v^5t=vt^5=v^3=1$. Thus $\tau\in<\sigma'>$, a contradiction.

      By a similar argument, we exclude the cases {$\Z/4\Z$} and $D_{2m}$, because for each $SmallGroup(6m,ID)$, there must be an element $\tau$ of order $2$ or $4$ which commutes with $\sigma'$.

      Finally, if $G'$ is conjugate to {$\Z/2\Z$} then there exists an element $\tau$ of order $2$ such that $\tau\sigma'\tau=\sigma'^{-1}$ and one can easily verify that such an element does not exists. { This follows immediately because $\tau=\tau^{-1}$ (being of order $2$), and since $\sigma'$ and $\sigma'^{-1}$ are not conjugate in $PGL_3(K)$}\\
      We conclude that $Aut(\tilde{\tilde{C}})$ should be cyclic (in particular, is commutative). Hence, it can not be of order $>3$ (otherwise; there must be an element $\tau\in Aut(\tilde{\tilde{C}})$ of order $>3$ which commutes with $\sigma'$, and by a previous argument, such elements do not exist).
  \item If $Aut(\tilde{\tilde{C}})$ fixed a triangle and neither a point nor a line is fixed, then it follows, by Harui \cite{Harui}, that $\tilde{\tilde{C}}$ is a descendant of the Fermat curve $F_6$ or the Klein curve $K_6$. The last case does not happen, because $Aut(K_6)$ does not have elements of order 3 whose Jordan form is the the same as $\sigma'$ (i.e a homology).
Now, suppose that $\tilde{\tilde{C}}$ is a descendant of $F_6$ that
is, $\tilde{\tilde{C}}$ can be transformed (through $P$) into a
curve $\tilde{\tilde{C}}_P$ whose core is $X^6+Y^6+Z^6$. Then,
$P=[\alpha_1X+\alpha_2Y;\,\beta_1X+\beta_2Y;\,Z]$, since there are
only two sets of homologies in $Aut(F_6)$ of order 3 namely,
$\{[\omega X;Y;Z],\,[X;\omega Y;Z],\,[X;Y;\omega Z]\}$ and
$\{[\omega^2 X;Y;Z],\,[X;\omega^2 Y;Z],\, [X;Y;\omega^2 Z]\}$
(recall that the two sets are not conjugate in $PGL_{3}(K)$. Also,
elements of the first set are all conjugate inside $Aut(F_6)$ to
$[X;Y;\omega Z]$. So it suffices to consider the situation
$P^{-1}\sigma P=\lambda\sigma$). Now, $\tilde{\tilde{C}}_P$ has the
form
$$
\mu_0X^6+\mu_1Y^6+Z^6+\alpha_3(\alpha_1X+\alpha_2Y)^3Z^3+\mu_2X^5Y+\mu_3X^4Y^2+\mu_4X^3Y^3+\mu_5X^2Y^4+\mu_6XY^5,
$$
where $\mu_0:=\alpha _1 \beta _1 \left(\alpha
_1^4+\beta_1^4\right)(=1)$ and $\mu_1:=\alpha _2 \beta _2
\left(\alpha _2^4+\beta_2^4\right)(=1)$. In particular,
$(\alpha_1\beta_1)(\alpha_2\beta_2)\neq0$ therefore,
$[X;vZ;wY],\,[vZ;wY;X],\,[wY;vZ;X]$, and $[vZ;X;wY]\notin
Aut(\tilde{\tilde{C}}_P)$, because of the monomial $XY^2Z^3$.
Moreover, $[wY;X;vZ]\in Aut(\tilde{\tilde{C}}_P)$ only if
$\alpha_1=\alpha_2$ and $w=v^3=1$. Hence
$$\tilde{\tilde{C}}_P:\,Z^6+\alpha_3\alpha_1^3(X+Y)^3Z^3+\alpha
_1(X+Y) \left(\beta_1X+\beta_2Y \right)\left(\alpha_1^4(X+Y)^4
+\left(\beta _1X +\beta_2Y \right)^4\right).$$ Consequently,
$\beta_1=\beta_2$ (because we are assuming $[Y;X;vZ]\in
Aut(\tilde{\tilde{C}}_P)$), a contradiction to invertibility of $P$. Finally, if $[X,\xi_6^rY,\xi_6^{r'}Z]\in Aut(\tilde{\tilde{C}}_P)$
then $r=0$ and $2|r'$, since $\alpha_1\alpha_2\neq0$. That is,
$|Aut(\tilde{\tilde{C}}_P)|=3$, which was to be shown.
\end{enumerate}
\end{proof}
As a conclusion of the results that are introduced in this section,
we get the following result.
\begin{cor} The locus $\widetilde{M_{10}^{Pl}(\Z/3\Z)}$ is not
ES-Irreducible, and it has at least two irreducible components.
\end{cor}

\section{Positive characteristic}

{Fix a prime $p>0$ and let $\mathbb{K}$ be an algebraically closed
field of positive characteristic $p>0$}. {Denote by $M_{g,p}$ the
moduli space on smooth, genus $g$ curves over the field
$\mathbb{K}$, and similarly, {we define the loci $M_{g,p}(G)$,
$M_{g,p}^{Pl}(G)$ and $\widetilde{M_{g,p}^{Pl}(G)}$ over
$\mathbb{K}$, as we did for zero characteristic} }.

{Following the abuse of notation of \S3, we c}onsider
a non-singular, {degree $d$} plane curve $C$ in
$\mathbb{P}^2(\mathbb{K})$, and assume that the order of $Aut(C)$ is
coprime with $p$, $p\nmid d(d-1)$ {and} $p\geq 7${. Also, suppose
that} the order of $Aut(F_d)$ and $Aut(K_d)$ are coprime with $p$
where $F_d: X^d+Y^d+Z^d=0$ is the Fermat curve, and
$K_d:X^{d-1}Y+Y^{d-1}Z+Z^{d-1}X=0$ is the Klein curve. Then, all
techniques, which appeared in Harui \cite{Harui}, can be applied:
Hurwitz bound, Arakawa and Oiakawa inequalities and so on. In
particular, the
arguments of {the previous} sections hold.

Consider the $p$-torsion of the degree 0 Picard group of $C$, which
is a finitely generated $\mathbb{Z}/(p)$-module of dimension
$\gamma$ (always $\gamma\leq g$, where $g$ is the genus of $C$). We
call $\gamma$ the $p$-rank of $C$.

For a point $\mathcal{P}$ of $C$, denote by $Aut(C)_{\mathcal{P}}$
the subgroup of $Aut(C)$ that fixes the place $\mathcal{P}$.
\begin{lem} Assume that $Aut(C)_{\mathcal{P}}$ is prime to $p$ for any point $\mathcal{P}$
of $C$ and the $p$-rank of $C$ is trivial. Then $Aut(C)$ is prime to
$p$.
\end{lem}
\begin{proof} {Let} $\sigma\in Aut(C)$ {be} of order $p${. T}hen the
extension $\mathbb{K}(C)/\mathbb{K}(C)^{\sigma}$ is a finite
extension of degree $p$, and is unramified everywhere (because if it
ramifies at a place {$\mathcal{P}$} then $\sigma$ will be an element
of $Aut(C)_{\mathcal{P}}$ giving a contradiction). But, if
$\gamma=0$ (i.e. the $p$-rank is trivial for $C$) then, from
Deuring-Shafarevich formula \cite[Theorem 11.62]{Book}, we {get}
$\frac{\gamma-1}{\gamma'-1}=p$ where $\gamma'$ is the $p$-rank for
$C/<\sigma>$, which is impossible. Therefore, such extensions do not
exist.
\end{proof}

\begin{lem}\label{lem6.2} {Let $C$ be} a plane non-singular curve of degree $d\geq 4$.
If $p> (d-1)(d-2)+1$, then $Aut(C)_{\mathcal{P}}$ is coprime with
$p$ for any point $\mathcal{P}$ of the curve $C$.
\end{lem}
\begin{proof} By \cite[Theorem 11.78]{Book}, the maximal order of
the $p$-subgroup of $Aut(C)_{\mathcal{P}}$ is at most
$\frac{4p}{(p-1)^2} g^2$. Hence, with $g=\frac{(d-1)(d-2)}{2}$ and
assuming that $p>\frac{4p}{(p-1)^2} g^2$, we obtain the result.
\end{proof}

\begin{lem}\label{lem6.3} Let $C$ be a non-singular curve of genus $g\geq 2$ {that is} defined over an algebraic
closed field $\mathbb{K}$ of characteristic $p>0$. Suppose that $C$
has an unramified subcover of degree $p$, i.e. $\Phi:C\rightarrow
C'$ of degree $p$. Then, $C'$ has genus $\geq 2$, $g\equiv 1(mod\
p)$ and $\gamma\equiv 1(mod\ p)$. In particular, {one needs to
assume that $p<g$ for the existence of such subcover}.
\end{lem}
\begin{proof} The Hurwitz formula for $\Phi$ gives the equality
$(2g-2)=p(2g'-2)$, where $g'$ is the genus of $C'$. We have $g'\neq
0$ or $1$ because $g\geq 2$, therefore $g'\geq 2$ and $g-1\equiv
0(mod\ p)$. Now, consider {the} Deuring-Shafaravich formula, {which
could be read as $\gamma-1=p(\gamma'-1)$ in such unramified
extension}, where $\gamma'$ the $p$-rank of $C'$. If $\gamma=1$ then
there is nothing to prove and if $\gamma>1$ then the congruence is
clear. Finally, {the situation $\gamma=0$ does not occur}.
\end{proof}

\begin{cor} Let $C$ be a non-singular{, degree $d$} plane curve of genus $g\geq 2$ over an algebraic
closed field $\mathbb{K}$ of characteristic $p>0$. Suppose that
$p>(d-1)(d-2)+1>g${, t}hen the order of $Aut(C)$ is coprime with
$p$.
\end{cor}
\begin{proof} Suppose $\sigma\in Aut(C)$ of order $p$, then
$\mathbb{K}(C)/\mathbb{K}(C)^{\sigma}$ is a separable degree $p$
extension, and by Lemma \ref{lem6.2}, it is unramified everywhere. By
Lemma \ref{lem6.3}, we {conclude} that such extensions do not exist. 
\end{proof}

As a direct consequence of the above lemmas, and because all
techniques {of \cite{Harui} are} applicable when $Aut(C)$ is coprime
with $p$, then we obtain:
\begin{cor} Assume {that} $p>13${, t}hen the automorphism groups of the curves $\tilde{C}:
X^5+Y^5+Z^4X+\beta X^3Y^2=0$ and
$\tilde{\tilde{C}}:\,\,\,X^5+X(Z^4+Y^4)+\beta Y^2Z^3=0$ such that
$\beta\neq0$, are cyclic of order 4. Moreover, $\tilde{C}$ is not
isomorphic to $\tilde{\tilde{C}}$ for any choice of the parameters.
\end{cor}
\begin{proof} {We need only} to mention that the linear $g_2$-systems
for the immersion of the curve inside $\mathbb{P}^2$ are unique, up
to conjugation in $PGL_3(\mathbb{K})$ (see \cite[Lemma
11.28]{Book}){. Also,} the curves $\tilde{C}$ and
$\tilde{\tilde{C}}$ have cyclic covers of degree 4 with different
type of the cover, from Hurwitz equation. Therefore, they belong to
different irreducible components in the moduli space of genus 6
curves.
\end{proof}

\begin{cor} {The locus {$\widetilde{M_{6,p}^{Pl}(\Z/4\Z)}$} with $p>13$} has at least two
irreducible components.
\end{cor}

Similarly, we get the following results from {those} of \S4 {and
\S5},

\begin{cor} {The locus
{$\widetilde{M_{g,p}^{Pl}(\Z/(d-1)\Z)}$}
is not ES-Irreducible, and it has at least two strongly
equation components for any odd integer $d\geq 5$ such that $p>(d-1)(d-2)+1$}. In particular, it has at least two irreducible
components.
\end{cor}

{
\begin{cor} For $p>21$, the locus
{$\widetilde{M_{10,p}^{Pl}(\Z/3)}$} has at least two
irreducible components.
\end{cor}}

\end{document}